\documentclass[12pt]{amsart}
\usepackage{amssymb}
\usepackage{amsfonts}
\usepackage{amsmath, amsthm}
\usepackage{enumerate}
\usepackage{bbold}
\usepackage{bbm}
\usepackage{dsfont}
\newtheorem{thm}{Theorem}[section]
\newtheorem{prop}[thm]{Proposition}
\newtheorem{rem}[thm]{Remark}
\newtheorem{exam}[thm]{Example}
\newtheorem{defi}[thm]{Definition}
\numberwithin{equation}{section}

\begin{document}

\begin{center}
{\bf The statistically unbounded $\tau$-convergence on locally solid Riesz spaces}
\end{center}

\vspace{4mm}

\begin{center}
Abdullah AYDIN
\end{center}

\address{Department of Mathematics, Mu\c{s} Alparslan University, Mu\c{s}, Turkey.}
\email{{aaydin.aabdullah@gmail.come-mail}}

\subjclass[2010]{40A35, 46A40}
	
\keywords{Statistically $u_\tau$-convergence, statistically $u_\tau$-cauchy, locally solid Riesz space, order convergence}
	
\title{}
	
\begin{abstract}
A sequence $(x_n)$ in a locally solid Riesz space $(E,\tau)$ is said to be statistically unbounded $\tau$-convergent to $x\in E$ if, for every zero neighborhood $U$, $\frac{1}{n}\big\lvert\{k\leq n:\lvert x_k-x\rvert\wedge u\notin U\}\big\rvert\to 0$ as $n\to\infty$. In this paper, we introduce this concept and give the notions $st$-$u_\tau$-closed subset, $st$-$u_\tau$-Cauchy sequence, $st$-$u_\tau$-continuous and  $st$-$u_\tau$-complete locally solid vector lattice. Also, we give some relations between the order convergence and the $st$-$u_\tau$-convergence.
\end{abstract}
\maketitle	
	
\section{Preliminary and Introductory Facts}\label{S1}

Vector lattices and the statistical convergence are natural and efficient tools in the theory of functional analysis. A vector lattice is an ordered vector space that has many applications in measure theory, operator theory and applications in economics; see for example \cite{AB,ABPO,Za}, it was introduced by F. Riesz \cite{Riez}. On the other hand, the statistical convergence is a generalization of the ordinary convergence of a real sequence; see for example \cite{MK}. Studies related to this paper are done by Maddox \cite{M} where the statistical convergence was introduced in a topological vector space and by Albayrak and Pehlivan \cite{AP} in which the statistical convergence was introduced on locally solid vector lattice. Otherwise, the unbounded order convergence was defined for order complete Riesz spaces by H. Nakano \cite{N}. Recently, some papers have been studied on the unbounded convergence on vector lattices (see, for example [4-7]). However, as far as we know, the concept of unbounded order convergence related to the statistical convergence has not been done before. In this paper, we aim to introduce this concept on locally solid vector lattice. 

Now, let give some basic notations and terminologies that will be used in this paper. Every linear topology $\tau$ on a vector space $E$ has a base $\mathcal{N}$ for the zero neighborhoods satisfying the following four properties; for each $V\in \mathcal{N}$, we have $\lambda V\subseteq V$ for all scalar $\lvert \lambda\rvert\leq 1$; for any $V_1,V_2\in \mathcal{N}$ there is another $V\in \mathcal{N}$ such that $V\subseteq V_1\cap V_2$; for each $V\in \mathcal{N}$ there exists another $U\in \mathcal{N}$ with $U+U\subseteq V$; for any scalar $\lambda$ and each $V\in \mathcal{N}$, the set $\lambda V$ is also in $\mathcal{N}$; for much more detail see \cite{AB,ABPO}. In this article, unless otherwise, when we mention a zero neighborhood, it means that it always belongs to a base that holds the above properties. Let $E$ be a real-valued vector space. If there is an order relation "$\leq$" on $E$, i.e., it is antisymmetric, reflexive and transitive, then $E$ is called ordered vector space whenever the following conditions hold: for every $x,y\in E$ such that $x\leq y$, we have $x+z\leq y+z$ and $\alpha x\leq\alpha y$ for all $z\in E$ and $\alpha \in \mathbb{R}_+$. An ordered vector space $E$ is called Riesz space or vector lattice if, for any two vectors $x,y\in E$, the infimum $x\wedge y=\inf\{x,y\}$ and the supremum $x\vee y=\sup\{x,y\}$ exist in $E$. Let $E$ be a vector lattice. Then, for any $x\in E$, the positive part of $x$ is $x^+:=x\vee0$, the negative part of $x$ is $x^-:=(-x)\vee0$ and absolute value of $x$ is $|x|:=x\vee (-x)$. Moreover, any two elements $x, \ y$ in a vector lattice is called disjoint whenever $\lvert x\rvert\wedge\lvert y\rvert=0$. An operator $T$ between two vector lattice $E$ and $F$ is said to be a lattice homomorphism or Riesz homomorphism whenever $T(x\vee y)=T(x)\vee T(y)$ holds for all $x,y\in E$. A vector lattice is called order complete if every nonempty bounded above subset has a supremum (or, equivalently, whenever every nonempty bounded below subset has an infimum). A vector lattice is order complete iff $0\leq x_n\uparrow\leq x$ implies the existence of $\sup{x_n}$; for much more detail information, see \cite{AB,ABPO,Riez,Za}.

Recall that a subset $A$ of a vector lattice $E$ is called solid if, for each $x\in A$ and $y\in E$,  $|y|\leq|x|$ implies $y\in A$. A solid vector subspace of a vector lattice is referred to as an ideal. An order closed ideal is called a band. Let $E$ be vector lattice $E$ and $\tau$ be a linear topology on it. Then $(E,\tau)$ is said a {\em locally solid vector lattice} (or, {\em locally solid Riesz space}) if $\tau$ has a base which consists of solid sets; for more details on these notions, see \cite{AB,ABPO,AAydn3,Za}. Recall that a net $(x_\alpha)_{\alpha\in A}$ in a vector lattice $X$ is {\em order convergent} to $x\in X$, if there exists another net $(y_\beta)_{\beta\in B}$ satisfying $y_\beta \downarrow 0$, and any $\beta\in B$, there exists $\alpha_\beta\in A$ such that $\lvert x_\alpha-x\rvert\leq y_\beta$ for all $\alpha\geq\alpha_\beta$. In this case, we write $x_\alpha\xrightarrow{0} x$. In a vector lattice $X$, a net $(x_\alpha)$ is unbounded order convergent to $x\in X$ if $\lvert x_\alpha-x\rvert\wedge u\xrightarrow{o} 0$ for every $u\in X_+$; see for example \cite{AAydn,N}. Also, these notions can be written for sequence. A vector lattice $E$ is called \textit{Archimedean} whenever $\frac{1}{n}x\downarrow 0$ holds in $E$ for each $x\in E_+$. In this article, unless otherwise, all vector lattices are assumed to be real and Archimedean.

Consider a  subset $K$ of the set $\mathbb{N}$ of all natural numbers. Let's define a new set $K_n=\{k\in K:k\leq n\}$. Then we denote $\lvert K_n\rvert$ for the cardinality of that the set $K_n$. If the limit of $\mu(K)=\lim\limits_{n\to\infty}\lvert K_n\rvert/n$ exists then $\mu(K)$ is called the asymptotic density of the set $K$. Let $X$ be a topological space and $(x_n)$ be a sequence in $X$. Then $(x_n)$ is said to be statistically convergent to $x\in X$ whenever, for each neighborhood $U$ of $x$, we have $\mu\big(\{n\in\mathbb{N}:x_n\notin U\}\big)=0$; see for example \cite{AP,M,MK}. Similarly, a sequence $(x_n)$ in a locally solid Riesz space $(E,\tau)$ is said to be statistically $\tau$-convergent to $x\in E$ if it is provided that, for every $\tau$-neighborhood $U$ of zero, $\lim\limits_{n\to\infty}\frac{1}{n}\big\lvert\{k\leq n:(x_k-x)\notin U\}\big\rvert=0$ holds; see \cite{AP}. Motivated by the above definitions, we give the following notion.
\begin{defi}
Let $(E,\tau)$ be a locally solid vector lattice and $(x_n)$ be a sequence in $E$. Then $(x_n)$ is said to be statistically unbounded $\tau$-convergent to $x\in E$ if, for every zero neighborhood $U\in\mathcal{N}$, it satisfies the following limit 
$$
\lim\limits_{n\to\infty}\frac{1}{n}\big\lvert\{k\leq n:\lvert x_k-x\rvert\wedge u\notin U\}\big\rvert=0 \eqno(1)
$$
for all $u\in E_+$. We abbreviate this convergence as $x_\alpha\xrightarrow{st-u_\tau}x$, or shortly, we say that $(x_n)$ $st$-$u_\tau$-converges to $x$. 
\end{defi}

Let's take $K_n=\{k\leq n:\lvert x_k-x\rvert\wedge u\notin U\}$ for arbitrary zero neighborhood $U$ in any locally solid vector lattice $E$. Then $(x_n)$ is statistically unbounded $\tau$-convergent to $x\in E$ whenever $\mu(K_n)=0$ in $(1)$ for every $u\in E_+$. In this paper, we use the following fact without giving the reference, and so we shall keep in mind it.
\begin{rem}\
Let $(E,\tau)$ be a locally solid Riesz space and $U$ be a zero neighborhood. If $y\leq x$ for any $x,y\in E$ then we have that $y\notin U$ implies $x\notin U$. Indeed, assume $y\notin U$ and $x\in U$. Then we have $y\in U$ because $U$ is solid and $y\leq x$. So, we have a contradict, and so we get $x\notin U$.
\end{rem}

\begin{exam}
Let's consider the locally solid Riesz space $(c_0,\lVert\cdot\rVert)$ the set of null sequences with the supremum norm $\lVert\cdot\rVert$on $c_0$. Consider the sequence $(e_n)$ of the standard unit vectors in $c_0$. Then the base $\mathcal{N}$ consists of the zero neighborhoods $U_r=\{x\in c_0:\lVert x\rVert<r\}$, where $r$ is a positive real number. Thus, we get that $e_n\xrightarrow{st-u_\tau}0$ in $c_0$. Indeed, for arbitrary zero neighborhood $U$, there exist some $U_r$ in $\mathcal{N}$ such that $U_r\subseteq U$, and also, $K_n=\{n\in \mathbb{N}:\lvert x_n-0\lvert\wedge u\in U_r\}=\{n\in \mathbb{N}:\lVert e_n\wedge u\rVert< r\}$ for all $u\in E_+$. It can be seen that $\mu(K_n)=1$. Thus, we have $e_n\xrightarrow{st-u_\tau}0$.
\end{exam}

\begin{rem}\label{basic remark}
The statistically $\tau$-convergence implies statistically unbounded $\tau$-convergence. Indeed, for arbitrary zero neighborhood $U$ and fixed $u\in E_+$, by applying the formula $\lvert x_k-x\rvert\wedge u\leq\lvert x_k-x\rvert$, we can  see that $\lvert x_k-x\rvert\wedge u\notin U$ implies $\lvert x_k-x\rvert\notin U$. Otherwise, it contradicts with being solid of $U$. Hence, $\{k\leq n:(\lvert x_k-x\rvert\wedge u)\notin U\}\subseteq\{k\leq n:(x_k-x)\notin U\}$ holds for all $u\in E_+$. It means that $st$-$\tau$-convergence implies $st$-$u_\tau$-convergence
\end{rem}
For the converse of Remark \ref{basic remark}, we give Theorem \ref{the converse theorem} in the next section.

\section{Main Results}\label{S2}

The lattice operations are uniformly continuous in locally solid vector lattice; see for example \cite[Thm.2.17.]{AB}. Similarly, the lattice operations are continuous in the following sense.
\begin{thm}\label{st are continuous}
Let $(x_n)$ and $(y_n)$ be two sequences in a locally solid vector lattice $(E,\tau)$. If $x_n\xrightarrow{st-u_\tau}x$ and $y_n\xrightarrow{st-u_\tau}y$ then $(x_n\vee y_n)\xrightarrow{st-u_\tau} x\vee y$.
\end{thm}

\begin{proof}
Let $U$ be an arbitrary zero neighborhood in $E$. Thus, there is another zero neighborhood $V$ in $\mathcal{N}$ such that $V+V\subseteq U$. Consider the set $K_{yn}=\{n\in\mathbb{N}:\lvert y_n-y\rvert\wedge u \in V\}$ and $K_{xn}=\{n\in\mathbb{N}:\lvert x_n-x\rvert\wedge u\in V\}$ for every $u\in E_+$. Then $K_{xn}=K_{yn}=1$ because of $x_n\xrightarrow{st-u_\tau}x$ and $y_n\xrightarrow{st-u_\tau}y$.
	
On the other hand, by applying \cite[Thm.1.9.(2)]{ABPO}, i.e., $\lvert a\vee b-b\vee c\rvert\leq \lvert a-c\rvert$ for any vectors in a vector lattice, we can get
\begin{eqnarray*}
\lvert x_n\vee y_n-x\vee y\rvert\wedge u&=&\lvert x_n\vee y_n-y_n\vee x+y_n\vee x-x\vee y\rvert\wedge u\\&\leq& \lvert x_n\vee y_n-y_n\vee x\rvert\wedge u+\lvert y_n\vee x-x\vee y\rvert\wedge u\\&\leq& \lvert x_n-x\rvert\wedge u+\lvert y_n-y\rvert\wedge u
\end{eqnarray*}
for all $u\in E_+$. So, we have $\lvert x_n\vee y_n-x\vee y\rvert\wedge u\in U$ for all $u\in E_+$ and for each $n\in \mathbb{N}$ because of $\lvert x_n-x\rvert \wedge u+\lvert y_n-y\rvert \wedge u\in V+V\subseteq U$. Thus, we have $\mu\big(\{n\in\mathbb{N}:\lvert x_n\vee y_n-x\vee y\rvert\wedge u \in U\}\big)=1$ for all $u\in E_+$. That is, we get $(x_n\vee y_n)\xrightarrow{st-u_\tau} x\vee y$.  
\end{proof}

By using the $st$-$u_\tau$-convergence, we can define the $st$-$u_\tau$-closed subset. 
\begin{defi}
Let $(E,\tau)$ be a locally solid vector lattice and $A$ be a subset of $E$. Then, $A$ is called $st$-$u_\tau$-closed subset in $E$ if, for any sequence $(a_n)$ in $A$ which is $st$-$u_\tau$-convergent to $a\in E$, it holds $a\in A$.
\end{defi}

\begin{prop}\label{E_+ is closed}
The positive cone $E_+=\{x\in E:0\leq x\}$ of a locally solid Riesz space $(E,\tau)$ is $st$-$u_\tau$-closed.
\end{prop}

\begin{proof}
Assume $(x_n)$ is a sequence in $E_+$ such that it $st$-$u_\tau$-converges $x\in E$. By applying Theorem \ref{st are continuous}, $x_n^+=x_n\vee0\xrightarrow{st-u_\tau}x\vee0=x^+\geq0$. Thus, we get $x\in E_+$.
\end{proof}

We continue with several basic results which are motivated by their analogies from vector lattice theory.
\begin{thm}\label{basic remarkasic properties of st convergence}
Let $x_n\xrightarrow{st-u_\tau}x$ and $y_n\xrightarrow{st-u_\tau}y$ in a locally solid Riesz space. Then we have the following facts:
\begin{enumerate}
\item[(i)] $x_n\xrightarrow{st-u_\tau}x$ iff  $(x_n-x)\xrightarrow{st-u_\tau}0$ iff $\lvert x_n-x\rvert\xrightarrow{st-u_\tau}0$;
\item[(ii)] $rx_n+sy_n\xrightarrow{st-u_\tau}rx+sy$ for any $r,s\in{\mathbb R}$;
\item[(iii)] if $x_n\ge y_n$ for all $n\in\mathbb{N}$ then $x\ge y$;
\item[(iv)] $x_{n_k}\xrightarrow{st-u_\tau}x$ for any subsequence $(x_{n_k})$ of $x_n$;
\item[(v)] $\lvert x_n\rvert\xrightarrow{st-u_\tau} \lvert x\rvert$;
\item[(vi)] if $\tau$ Hausdroff topology, $x_n\xrightarrow{st-u_\tau}x$ and $x_n\xrightarrow{st-u_\tau}y$ then $x=y$.
\end{enumerate}	
\end{thm} 

\begin{proof}
The $(i)$ comes directly from the definition of the $st$-$u_\tau$-convergence.
	
$(ii)$ Firstly, we show that $rx_n\xrightarrow{st-u_\tau}rx$ for every $r\in\mathbb{R}$. Let $U$ be an arbitrary zero neighborhood and $r$ be a real number. Thus, we have $\mu\big(\{n\in\mathbb{N}:\lvert x_n-x\rvert\wedge u\in U\}\big)=1$ for all $u\in E_+$. If $\lvert r\rvert\leq 1$ then, by applying the properties of $\mathcal{N}$, we can get $\lvert r\rvert(\lvert y_n-y\rvert\wedge u)=\lvert rx_n-rx\rvert\wedge(\lvert r\rvert u)\leq\lvert x_n-x\rvert\wedge u \in U$. In special, if we take $u$ as $\lvert r\rvert u$ then $\lvert rx_n-rx\rvert\wedge u\in U$. So, we get $\mu\big(\{n\in\mathbb{N}:\lvert rx_n-rx\rvert\wedge u\in U\}\big)=1$ for all $\lvert r\rvert\leq 1$ and every $u\in E_+$ because of $\{n\in\mathbb{N}:\lvert x_n-x\rvert\wedge u\in U\}\subseteq\{n\in\mathbb{N}:\lvert rx_n-rx\rvert\wedge u\in U\}$ for each zero neighborhood $U$. It means that $rx_n\xrightarrow{st-u_\tau}rx$ holds for each $\lvert r\rvert\leq 1$. 
	
Next, for $\lvert r\rvert> 1$, it follows from the third property of $\mathcal{N}$ that there is another $V\in\mathcal{N}$ such that $\{V+V+\dots+V\}_{k}\subseteq U$, where $k$ is the smallest integer greater or equal $r$. Then, by the inequality $\lvert rx_n-rx\rvert\wedge u=(\lvert r\rvert\lvert x_n-x\rvert)\wedge u\leq (\lvert k\rvert\lvert x_n-x\rvert)\wedge u=\lvert kx_n-kx\rvert\wedge u\in \{V+V+\dots+V\}_{k}\subseteq U$, we can get that $\lvert rx_n-rx\rvert\wedge u\in U$ for each $u\in E_+$ and for all $n\in\mathbb{N}$ because $U$ is a solid subset. Therefore, we get $\mu\big(\{n\in\mathbb{N}:\lvert rx_n-rx\rvert\wedge u\in U\}\big)=1$ for each $u\in E_+$, i.e., $rx_n\xrightarrow{st-u_\tau}rx$. 
	
Now, we show $x_n+y_n\xrightarrow{st-u_\tau}x+y$. Take $U$ an arbitrary zero neighborhood. By applying the properties of $\mathcal{N}$, there exists another $V\in \mathcal{N}$ with $V+V\subseteq U$. Since $x_n\xrightarrow{st-u_\tau}x$ and $y_n\xrightarrow{st-u_\tau}y$, we have $\mu\big(\{n\in\mathbb{N}:\lvert x_n-x\rvert\wedge u\in U\}\big)=1$ and $\mu\big(\{n\in\mathbb{N}:\lvert y_n-y\rvert\wedge u\in U\}\big)=1$ for all $u\in E_+$. By the formula $\lvert(x_n+y_n)-(x+y)\rvert\wedge u\leq \lvert(x_n-x)\rvert\wedge u+\lvert(y_n+y)\rvert\wedge u\in V+V\subseteq U$ for every $u\in E_+$. So, we get $\mu\big(\{n\in\mathbb{N}:\lvert(x_n+y_n)-(x+y)\rvert\wedge u\in U\}\big)=1$. That is, $x_n+y_n\xrightarrow{st-u_\tau}x+y$.\\
	
$(iii)$ Suppose that $y_n\leq x_n$ holds for all $n\in\mathbb{N}$. Then we can get $0\leq x_n-y_n \in E_+$ for each $n\in\mathbb{N}$. By using $(i)$ and applying Proposition \ref{E_+ is closed}, we have $x_n-y_n\xrightarrow{st-u_\tau}x-y\in E_+$ because of $(x_n-y_n)\in E_+$. Thus, we get $x-y\geq 0$, i.e., $x\geq y$. \\
	
$(iv)$ For arbitrary zero neighborhood $U$, it follows from $\{k\in\mathbb{N}:\lvert x_{n_k}-x\rvert\wedge u\notin U\}\subseteq\{n\in\mathbb{N}:\lvert x_n-x\rvert\wedge u\notin U\}$ for all $u\in E_+$ and $x_n\xrightarrow{st-u_\tau}x$ that $\mu\big(\{k\in\mathbb{N}:\lvert x_{n_k}-x\rvert\wedge u\notin U\}\big)=0$ for every $u\in E_+$ i.e.,  $x_{n_k}\xrightarrow{st-u_\tau}x$ holds.\\
	
$(v)$ By applying the formula $\lvert x\rvert=x\vee(-x)$, and by using $(i)$ and Theorem \ref{st are continuous}, one can get the desired result.\\
	
$(vi)$ Let $U$ be a zero neighborhood. Take $K_n=\{n\in\mathbb{N}:\lvert x_n-x\rvert\wedge u\in U\}$ and $L_n=\{n\in\mathbb{N}:\lvert x_n-y\rvert\wedge u\in U\}$ for $u\in E_+$. Since  $x_n\xrightarrow{st-u_\tau}x$ and $x_n\xrightarrow{st-u_\tau}y$, we have $\mu(K_n)=\mu(L_n)=1$. On another hand, by applying the properties of $\mathcal{N}$, there is another $V\in \mathcal{N}$ such that $V+V\subseteq U$. Consider the inequality $\lvert x-y\rvert\wedge u\leq\lvert x-x_n\rvert\wedge u+\lvert x_n-y\rvert\wedge u\in V+V\subseteq U$ for all $u\in E_+$ and for each $n\in\mathbb{N}$. Thus, for arbitrary zero neighborhood $U$, we have $\lvert x-y\rvert\wedge u\in U$ for every $u\in E_+$. But, since $E$ is a Hausdorff space, the intersection of all zero neighborhood is the singleton $0$. That is, $\lvert x-y\rvert\wedge u=0$ for all $u\in E_+$, and so, we get $x=y$.
\end{proof}

The following proposition is an $st$-$u_\tau$-version of \cite[Lem.2.5.(ii)]{AAydn2} and \cite[Thm.2.8.]{AAydn}.
\begin{prop}\label{motone and st implies order}
Any monotone $st$-$u_\tau$-convergent sequence in a locally solid vector lattice order converges to its $st$-$u_\tau$-limit.
\end{prop}

\begin{proof}
It is enough to show that if a sequence $(x_n)\uparrow$ in a locally solid vector lattice $(E,\tau)$ and $x_n\xrightarrow{st-u_\tau}x$ then $x_n\uparrow x$. Let's fix arbitrary $n\in \mathbb{N}$. Then $x_m-x_n\in E_+$ for every $m\geq n$. By using Proposition \ref{E_+ is closed} and Theorem \ref{basic remarkasic properties of st convergence}$(i)$, we get $x_m-x_n\xrightarrow{st-u_\tau}x-x_n\in E_+$ as $m\to\infty$. Therefore, $x\geq x_n$ for any $n$. Thus, since $n$ is arbitrary, $x$ is an upper bound of $(x_n)$. Assume $y$ is another upper bound of $(x_n)$, i.e., $y\geq x_n$ for all $n\in \mathbb{N}$. Then, again by Proposition \ref{E_+ is closed}, we have $y-x_n\xrightarrow{st-u_\tau} y-x\in E_+$, or $y\ge x$. Thus, we get the desired result $x_n \uparrow x$.
\end{proof}

In Remark \ref{basic remark}, we show that the statistically $\tau$-convergence implies statistically unbounded $\tau$-convergence. For the converse, we give the following theorem.
\begin{thm}\label{the converse theorem}
Let $(x_n)$ be a monotone sequence in a locally solid Riesz space $(E,\tau)$. Then the statistically unbounded $\tau$-convergence of $(x_n)$ implies the statistically $\tau$-convergence of it.
\end{thm}

\begin{proof}
Without loss of generality, we may assume that $(x_n)$ positive and increasing sequence in $E_+$. It follows from Proposition \ref{motone and st implies order} that $x_n\uparrow x$ for some $x\in E$ because of $x_n\xrightarrow{st-u_\tau}x$. Hence, we have $0\leq x-x_n\leq x$ for all $n\in\mathbb{N}$. On the other hand, for each $u\in E_+$ and for any arbitrary zero neighborhood $U$, we have
$$
\mu\big(\{n\in\mathbb{N}:\lvert x_n-x\rvert\wedge u\in U\}\big)=1.
$$ 
In particular, for $u=x$, we obtain that
$$
\{n\in\mathbb{N}:\lvert x_n-x\rvert\wedge u\in U\}=\{n\in\mathbb{N}: (x-x_n)\in U\}. \eqno(2)
$$
Thus, we get $\mu\big(\{n\in\mathbb{N}:(x-x_n)\in U\}\big)=1$ from $(2)$, i.e., $(x_n)$  statistically $\tau$-converges to $x$.
\end{proof}

Recall that a band $B$ in a vector lattice $E$ is called a projection band whenever it satisfies $E=B\oplus B^d$, where $B^d=\{x\in E:\lvert x\rvert\wedge\lvert b\rvert=0,\ \text{for all}\ b\in B \}$ is disjoint complement set of $B$.
\begin{prop}
Let $(E,\tau)$ be a locally solid Riesz space and $B$ be a projection band of $E$. If $x_n\xrightarrow{st-u_\tau}x$ in $E$ then  $P_B(x_n)\xrightarrow{st-u_\tau} P_B(x)$ in both $E$ and $B$, where $P_B$ is the corresponding order projection of $B$. 
\end{prop}

\begin{proof}
Let $U$ be an arbitrary zero neighborhood. Then we have $\mu\big(\{n\in\mathbb{N}:(\lvert x_n-x\rvert\wedge u)\in U\}\big)=1$ because of $x_n\xrightarrow{st-u_\tau}x$. It is known that every order projection is an order continuous lattice homomorphism; see \cite[p.94]{ABPO}. Thus, $P_B$ is a lattice homomorphism and $0\leq P_B\leq I$; see \cite[Thm.1.44.]{ABPO}. Now, by applying \cite[Thm.2.14.]{ABPO}, we can get $\lvert P_B(x_n)-P_B(x)\rvert\wedge u=P_B(\lvert x_n-x\rvert)\wedge u\leq \lvert x_\alpha-x\rvert\wedge u$ for every $u\in E_+$. Then it follows easily from $\mu\big(\{n\in\mathbb{N}:(\lvert P_B(x_n)-P_B(x)\rvert\wedge u)\in U\}\big)=1$ that $P_B(x_\alpha)\xrightarrow{st-u_\tau} P_B(x)$ in $E$ and $B$.
\end{proof}

We continue with several basic notions in locally solid vector lattice concerning the $st$-$u_\tau$-convergence, which are motivated by their analogies from vector lattice theory.
\begin{defi}
Let $(E,\tau)$ be a locally solid vector lattice. Then 
\begin{enumerate}
\item[(1)] a sequence $(x_n)$ in $E$ is said to be $st$-$u_\tau$-Cauchy if the sequence $(x_m-x_n)_{(m,n)\in\mathbb{N}\times\mathbb{N}}$ $st$-$u_\tau$-converges to zero, i.e., for each zero neighborhood $U$, $\mu\big(\{n\in\mathbb{N}:(\lvert x_m-x_n\rvert\wedge u)\notin U\}\big)=0$ for all $u\in E_+$;
\item[(2)] $E$ is called order $st$-$u_\tau$-continuous if $x_n\xrightarrow{o}0$ implies $x_n\xrightarrow{st-u_\tau}0$;
\item[(3)] $E$ is called $st$-$u_\tau$-complete if every $st$-$u_\tau$-Cauchy sequence in $E$ is $st$-$u_\tau$-convergent.
\end{enumerate}	
\end{defi}
The next proposition follows from the basic definitions and results, so its proof is omitted.
\begin{prop}
If a sequence $(x_n)$ in a locally solid Riesz space $(E,\tau)$ is statistically $u_\tau$-convergent then it is statistically $u_\tau$-Cauchy.
\end{prop}

\begin{thm}\label{order continuous and stu}
Let $(E,\tau)$ be a locally solid vector lattice and $(x_n)$ be a sequence in $E$. Then $E$ is order $st$-$u_\tau$-continuous iff $x_n\downarrow 0$ implies $x_n\xrightarrow{st-u_\tau}0$ in $E$.	
\end{thm}

\begin{proof}
The implication is obvious, so we show the converse direction. Suppose $x_n\downarrow 0$ implies $x_n\xrightarrow{st-u_\tau}0$ in $E$. Take $x_n\xrightarrow{o}0$ in $E$. We show $x_n\xrightarrow{st-u_\tau}0$. By the order convergence of $(x_n)$, there is another sequence $(z_n)\downarrow 0$ in $E$ such that, for each $n$, there exists $n_k\in \mathbb{N}$ so that $\lvert x_n\rvert\leq z_n$ for all $n\geq n_k$. Thus, we have $z_n\xrightarrow{st-u_\tau}0$ because of $z_n\downarrow 0$. Therefore, for arbitrary zero neighborhood $U$, we have $\mu\big(\{n\in\mathbb{N}:(\lvert z_n\rvert\wedge u)\in U\}\big)=1$ for all $u\in E_+$. Since $U$ is solid and $\lvert x_n\rvert\leq z_n$ for all $n\geq n_k$, we have $(\lvert x_n\rvert\wedge u)\in U$ for every $n\geq n_k$ and $u\in E_+$. As a result, $\mu\big(\{n\in\mathbb{N}:(\lvert x_n\rvert\wedge u)\in U\}\big)=1$ for arbitrary $U$ and for all $u\in E_+$.
\end{proof}

In the case of $st$-$u_\tau$-complete locally solid vector lattice, we have the following result.
\begin{thm}\label{op-contchar}
For an $st$-$u_\tau$-complete locally solid vector lattice $(E,\tau)$, the following statements are equivalent:
\begin{enumerate}
\item[(i)] $E$ is order $st$-$u_\tau$-continuous;
\item[(ii)] if $0\leq x_n\uparrow\leq x$ holds in $E$ then $(x_n)$ is an $st$-$u_\tau$-Cauchy sequence;
\item[(iii)] $x_n\downarrow 0$ in $E$ implies $x_n\xrightarrow{st-u_\tau}0$.	
\end{enumerate}	
\end{thm}

\begin{proof} 
$(i)\Rightarrow(ii)$: Let $(x_n)$ be a positive increasing sequence in $E_+$. Then, by using \cite[Lem.4.8.]{ABPO}, we have another sequence $E$ such that $(y_k-x_n)_{(k,n)\in\mathbb{N}\times\mathbb{N}}\downarrow 0$. Thus, it follow from assumption that we have $(y_k-x_n)_{(k,n)\in\mathbb{N}\times\mathbb{N}}\xrightarrow{st-u_\tau} 0$. Then, for any zero neighborhood $U$, we have  $\mu\big(\{n\in\mathbb{N}:(\lvert y_k-x_n\rvert\wedge u)\in U\}\big)=1$ for all $u\in E_+$ and for every $k,n\in \mathbb{N}$. By properties of $\mathcal{N}$, there is another zero neighborhood $V$ such that $V+V\subseteq U$. So, by using the inequality $\lvert x_n-x_m\rvert\wedge u\leq \lvert x_n-y_k\rvert\wedge u+\lvert y_k-x_m\rvert\wedge u\in V+V\subseteq U$, we have $\lvert x_n-x_m\rvert\wedge u\in U$. As a result, we get $\mu\big(\{n\in\mathbb{N}:(\lvert x_n-x_m\rvert\wedge u)\in U\}\big)=1$ for all $u\in E_+$. It means that $(x_n)$ is an $st$-$u_\tau$-Cauchy sequence. 
	
$(ii)\Rightarrow(iii)$: Assume that $x_n\downarrow 0$ is a sequence in $E$. Let's fix an arbitrary index $n_0$. Then, we have $x_n\leq x_{n_0}$ whenever $n\geq n_0$. So, we get $0\leq(x_{n_0}-x_n)_{n\geq n_0}\uparrow\leq x_{n_0}$. Thus, we can apply the condition $(ii)$, and so the sequence $(x_{n_0}-x_n)_{n\geq n_0}$ is $st$-$u_\tau$-Cauchy, i.e., $(x_n-x_{n'})_{(n,n')\in\mathbb{N}\times\mathbb{N}}\xrightarrow{st-u_\tau}0$ as $n_0\leq n,n'\to \infty$. Now, it follows from the $st$-$u_\tau$-completeness of $E$ that there exists an element $x\in E$ such that $x_n\xrightarrow{st-u_\tau}x$ as $n_0\leq n\to\infty$. By Proposition \ref{motone and st implies order}, we have $x_\alpha\downarrow x$, and so it is clear $x=0$. As a result, we get $x_n\xrightarrow{st-u_\tau}0$.
	
$(iii)\Rightarrow(i)$: It is just the implication of Theorem \ref{order continuous and stu}.
\end{proof}

\begin{thm}
Let $(E,\tau)$ be an $st$-$u_\tau$-continuous and $st$-$u_\tau$-complete locally solid vector lattice. Then $E$ is order complete. 
\end{thm}

\begin{proof}
Let $(x_n)$ be a positive, increasing and bounded sequence by a vector $e\in E_+$. Then by applying Theorem \ref{op-contchar}$(ii)$, we see that $(x_n)$ is an $st$-$u_\tau$-Cauchy sequence. Then there is $x\in E$ such that $x_n\xrightarrow{st-u_\tau}x$ because $E$ is $st$-$u_\tau$-complete. It follows from Proposition \ref{motone and st implies order} that $x_n\uparrow x$, and so $E$ is order complete.
\end{proof}


\begin{thebibliography}{99}
\normalsize		

\bibitem{AB}
Aliprantis CD, Burkinshaw O.
Locally solid Riesz spaces with applications to economics.
Mathematical Surveys and Monographs, American Mathematical Society, 2003.

\bibitem{ABPO}
Aliprantis CD, Burkinshaw O.  
Positive Operators.
Orlando, Academic Press, 1985.

\bibitem{AP}
Albayrak H, Pehlivan S.
Statistical convergence and statistical continuity on locally solid Riesz.
Topology and its Applications 2012, 159: 1887-1893.

\bibitem{AAydn}
Ayd\i n A.
Unbounded $p_\tau$-Convergence in Vector Lattice Normed by Locally Solid Lattices.
Academic studies in Mathematic and Natural Sciences-2019/2, 118-134, IVPE, Cetinje-Montenegro, 2019. 

\bibitem{AAydn2}
Ayd\i n A.
Multiplicative order convergence in f-algebras.
In press Hacettepe Journal of Mathematics and Statistics, doi.org/10.15672/hujms.512073.

\bibitem{AAydn3}
Ayd\i n A.
Convergence via filter in locally solid Riesz spaces.
International Journal of Science and Research 2019, 9: 351-353.

\bibitem{AGG}
Ayd{\i}n A, Gorokhova SG, G\"{u}l H.
Nonstandard hulls of lattice-normed ordered vector spaces.
Turkish Journal of Mathematics 2018, 42: 155-163.

\bibitem{M}
Maddox IJ.
Statistical convergence in a locally convex space.
Mathematical Proceedings of the Cambridge Philosophical Society 1988, 104: 141–145.

\bibitem{MK}
Maio GD, Kocinac LDR.
Statistical convergence in topology.
Topology and its Applications 2008, 156: 28–45.

\bibitem{N}
Nakano H.
Ergodic theorems in semi-ordered linear spaces.
Annals of Mathematics 1948, 49: 538-556.

\bibitem{Riez}
Riesz F.
Sur la décomposition des opérations fonctionelles linéaires.
Bologna, Atti Del Congresso Internazionale Dei Mathematics, 1928.

\bibitem{Za}
Zaanen AC.
Riesz spaces, II.
Amsterdam, the Netherlands: North-Holland Publishing Co., 1983.
\end{thebibliography}
\end{document}